\documentclass[12pt]{amsart}

\usepackage{hyperref}
\usepackage{amssymb}
\usepackage{amsmath}
\usepackage{amsthm}
\usepackage{enumerate}
\usepackage{graphicx}
\usepackage{mathrsfs}
\usepackage{extarrows}
\usepackage{color}

\usepackage{setspace}

\theoremstyle{plain}
\newtheorem{thm}{Theorem}
\newtheorem{lem}[thm]{Lemma}
\newtheorem{prop}[thm]{Proposition}

\theoremstyle{definition}

\theoremstyle{remark}
\newtheorem{re}{Remark}

\begin{document}
	
\title[The Best constant for the almost uncentered maximal operator]
	{The Best constant for the almost uncentered maximal operator on radial decreasing functions}
	
\author{Wu-yi Pan}

\address{Key Laboratory of Computing and Stochastic Mathematics (Ministry of Education),
		School of Mathematics and Statistics, Hunan Normal University, Changsha, Hunan 410081, P.
		R. China}
	
\email{pwyyyds@163.com}

\date{\today}
	

\thanks{The research is supported in part by the NNSF of China (Nos. 11831007, 12071125)}	
		
	
\begin{abstract}
The purpose of this note is to find the least weak type $(1,1)$ bound for the almost uncentered maximal operator on radial decreasing functions.
\end{abstract}
	
\maketitle

Let $f$ be a non-negative locally
integrable function on $\mathbb{R}^d$ and let $\lambda$ be a real number in the unit interval $[0,1]$. The \emph{almost uncentred maximal function} of $f$ is defined by
\[
M_\lambda f\stackrel{\rm def}{=} \sup_{B:\lambda B\ni x} \frac{1}{|B|}\int_B fdx
\]
where set $\lambda B$ to
be the image of the ball $B$ under the homothety with the center of the $B$ and $|\cdot|$ denotes Lebesgue volume. With the increase of  control variables $\lambda$,  operators could continuous change from centered to uncentered  if we set $0B$ to be the center of $B$. The Hardy–Littlewood maximal theorem states that for $t>0$,
\begin{equation}\label{eq:1.1}
	|\{x\in \mathbb{R}^d:M_\lambda f(x)>t\}|\leq \frac{C(d)}{t}\Vert f\Vert_1
\end{equation}
where $C(d)$ depends only on the dimension $d$. Denote by $c_{d,\lambda}$ the best constant in \eqref{eq:1.1}.  The question of how large this number can been studied by several authors.  Here $c_{1,1}$ is equal to $2$ corresponding to a single Dirac delta. Compared with it, Melas \cite{Me03} determined the exact value of $c_{1,0}$ as $\frac{11+\sqrt{61}}{12}$  by a complicated argument. If  imposing on $f$ the
additional conditions of being radial, Men\'{a}rguez and Soria \cite{MS93} showed that $c_{d,0}\leq 4$ for all $d$.  Afterwards, Aldaz and Pérez Lázaro \cite{AP11} proved $c_{d,0}=1$ if proceeding to add the assumption of radially decreasing. 

In this short note, we prove that for all $0\leq \lambda\leq 1$, $c_{d,\lambda}=(1+\lambda)^d$ for considering the class of radial decreasing functions, that is

\begin{thm}\label{thm:1}
	Let $g\in L^1(\mathbb{R}^d)$ be a radial decreasing function.  Then for every $t>0$, \begin{equation*}
		|\{x\in \mathbb{R}^d:M_\lambda g(x)>t\}|\leq \frac{(1+\lambda)^d}{t}\Vert g\Vert_1.
	\end{equation*}
	Furthermore, the constant is sharp. 
\end{thm}

We will assume that each ball is closed in the rest of this note.
\begin{lem}
		Let $g\in L^1(\mathbb{R}^d)$ be a radial function. Then $M_\lambda g$ is also radial and has a form of the following expression$:$	\begin{equation*}
			M_\lambda g (x) = \sup_{(\alpha,\beta):\alpha\in [0,1];\lambda\beta+\alpha\geq1} \frac{1}{|B(\alpha x, \beta\Vert x\Vert)|}\int_{B(\alpha x, \beta\Vert x\Vert)} gdx.
		\end{equation*} 
\end{lem}
\begin{proof}
	Fixed an $x\in \mathbb{R}^d$, let $B$ be a ball such that $\lambda B\ni x$. Then there exists a rotate transformation $T$ about the origin such that the center of $T(B)$ lies on the line segment bounded by the origin and $x$. Clearly, $x\in T(B)$ and  $\frac{1}{|B|}\int_B gdx=\frac{1}{|T(B)|}\int_{T(B)} gdx$. Yet the general, we can assume that the supremum is taken over in such ball sub-collections. The process of calculating the expression is simple, so we ignore it.
\end{proof}
\begin{re}
	This result still holds if we consider the case of homogeneous manifolds equipping a invariant measure. But the next proof of Lemma \ref{le:3} strongly depends on the homogeneous property of Lebesgue measure. The reader may find the other approach to tackle this issue.

\end{re}
\begin{lem}\label{le:3}
		Let $g\in L^1(\mathbb{R}^d)$ be a radial decreasing function. Then
		\begin{equation*}
			M_0 g (x) = \sup_{r\in [\Vert x\Vert,2\Vert x\Vert]} \frac{1}{|B(x,r)|}\int_{B(x,r)} gdx.
		\end{equation*} 
\end{lem}
\begin{proof}
	The proof is similar in spirit to Aldaz et al.(See \cite[Lemma 2.1]{AP11}). We are working on the assumption that $x$ is a unit vector on the first horizontal axis. We only need to show 
	\begin{equation}\label{ineq:2.3}
		r^d|B(e_1,1)\cap B(0,t)|\geq |B(e_1,r)\cap B(0,t)|
	\end{equation} for all $t+r>1$ and $r\in [0,1]$. If this happens then it is easy to show  \begin{equation*}
	\frac{1}{|B(e_1,1)|}\int_{B(e_1,1)} gdx\geq\frac{1}{|B(e_1,r)|}\int_{B(e_1,r)} gdx.
\end{equation*} This implies
\begin{equation*}
	M_0 g (x) = \sup_{r\in [\Vert x\Vert,+\infty]} \frac{1}{|B(x,r)|}\int_{B(x,r)} gdx.
\end{equation*} The other half of the proof is obvious, we omit it.  We now turn to prove \eqref{ineq:2.3}. To this end,  we note that the left side of the formula \eqref{ineq:2.3} is equal to $|B(e_1,r)\cap B((1-r)e_1,rt)|$ and thus we have to prove $	B(e_1,r)\cap B(0,t)	\subset B((1-r)e_1,rt)\cap B(0,t)$. It is exactly as Aldaz et al. proved that 
\begin{equation}\label{ineq:2.4}
	B(e_1,r)\cap B(0,t)\subset B(\frac{(1+t^2-r^2)e_1}{2},rt).
\end{equation}
Note that $\Vert \frac{(1+t^2-r^2)e_1}{2} \Vert\geq \Vert (1-r)e_1 \Vert$ since $t>1-r\geq 0$. And thus \eqref{ineq:2.4} ensure that $B(e_1,r)\cap B(0,t)	\subset B((1-r)e_1,rt)\cap B(0,t)$ holds. That's what we need, so lemma follows.
\end{proof} 
Let us mention a consequence of the lemmas. 
\begin{prop}\label{pr:4}
	Let $g\in L^1(\mathbb{R}^d)$ be a radial decreasing function. Then 	\begin{equation*}
		M_\lambda g (x) = \sup_{(\alpha,\beta):\alpha\in [\beta,\beta+1];\lambda\beta+\alpha\geq1} \frac{1}{|B(\alpha x, \beta\Vert x\Vert)|}\int_{B(\alpha x, \beta\Vert x\Vert)} gdx.
	\end{equation*} 
\end{prop}
We are now in a position to prove the main result.
\begin{proof}[Proof of Theorem \ref{thm:1}.] 
	Suppose $\Vert g\Vert_1=1$. By Proposition \ref{pr:4}, we have
	\begin{equation*}
		M_\lambda g (x) \leq \sup_{(\alpha,\beta):\alpha\in [\beta,\beta+1];\lambda\beta+\alpha\geq1} \frac{1}{|B(\alpha x, \beta\Vert x\Vert)|}=\frac{1}{|B(\frac{x}{1+\lambda}, \frac{\Vert x\Vert}{1+\lambda})|}
	\end{equation*}
if $\Vert x\Vert=R>0$. Hence the weak $(1,1)$ estimation follows. Considering  $f=\frac{\chi_{B(0,r)}}{|B(0,r)|}$ for $r>0$ and applying the standard technique, we can easily prove that $(1+\lambda)^d$ is the best (for instance \cite{Ko15}).
\end{proof}


\begin{thebibliography}{999}
	
\bibitem[AP11]{AP11} J. M. Aldaz, J. Pérez Lázaro,
\textit{The best constant for the centered maximal operator on radial decreasing functions.} Math. Inequal. Appl., 14 (2011), no. 1, 173–179.
	
\bibitem[Ko15]{Ko15} D. Kosz, \textit{ On the discretization technique for the Hardy-Littlewood maximal operators.} Real Anal. Exchange, 41 (2016), no. 2, 287–292.


\bibitem[Me03]{Me03} A. D. Melas, The best constant for the centered Hardy-Littlewood maximal inequality, Ann. of Math.(2), 157, 2 (2003), 647–688. MR1973058

\bibitem[MS93]{MS93} M. T. Menárguez, F. Soria, On the maximal operator associated to a convex body in Rn, Collect. Math., 43, 3 (1992), 243–251 (1993).  
 



\end{thebibliography}
\end{document}